\documentclass[12pt]{amsart}
\usepackage{amssymb,amsmath,enumerate,array,
color,makecell,geometry,multirow,
hyperref,bm}
\usepackage[english]{babel}
\usepackage[all]{xy}
\newcommand{\U}{\mathrm{U}}
\allowdisplaybreaks
\def\Stone#1{\fbox{\makebox[13mm]{\strut#1}}\kern2pt}
\newtheorem{theorem}{Theorem}[section]

\newtheorem{corollary}[theorem]{Corollary}
\newtheorem{proposition}[theorem]{Proposition}
\newtheorem{example}[theorem]{Example}
\newtheorem{remark}[theorem]{Remark}

\newtheorem{definition}[theorem]{Definition}

\begin{document}

\title[Gelfand--Tsetlin bases for representations of $U_{q}(\mathfrak{gl}_n)$ at roots of unity]
{Gelfand--Tsetlin bases for representations of $\U_{q}(\mathfrak{gl}_n)$ at roots of unity}
\author{Hao Chang}
\address{School of Mathematics and Statistics, Central China Normal University, Wuhan, Hubei 430079, China}
\email{chang@ccnu.edu.cn }
\author{Ruiying Hou}
\address{School of Mathematics and Statistics, Central China Normal University, Wuhan, Hubei 430079, China}
\email{houry1998@163.com }
\author{Jian Zhang}
\address{School of Mathematics and Statistics, Central China Normal University, Wuhan, Hubei 430079, China}
\email{jzhang@ccnu.edu.cn}

\begin{abstract}
In this paper, we construct Gelfand--Tsetlin bases for representations of $\mathrm U_q(\mathfrak{gl}_n)$ at roots of unity. As an application, we obtain explicit bases for certain baby Verma modules of $\mathrm U_q(\mathfrak{gl}_3)$ and determine their composition series.
\end{abstract}
\maketitle
	
\section{Introduction}
The quantum group $\U_q(\mathfrak{g})$ is a deformation of the universal enveloping algebra $\U(\mathfrak{g})$ of a simple Lie algebra $\mathfrak{g}$.
The representation theory of quantum groups has been a remarkably active area of mathematics,
and its deep connections to theoretical physics have been increasingly revealed over the past decades.
A key instance is
Chern-Simons theory: Witten \cite{WE1989} showed that it produces link invariants for compact simple Lie groups,
and Reshetikhin and Turaev \cite{RT1991} constructed these invariants rigorously from the quantum group at roots of unity.

For generic $q$, the representation theory of $\U_q(\mathfrak{g})$ closely parallels that of its classical counterpart $\U(\mathfrak{g})$.
When $q$ is a root of unity, however, the quantum group admits a large central subalgebra, and the category of finite-dimensional representations is no longer semisimple.
The representation theory of $\U_q(\mathfrak{g})$ at roots of unity is considerably richer than that for generic $q$.
A series of fundamental results due to De Concini, Kac, Procesi \cite{DK1990,DKC1992,DK1992}, as well as Lusztig \cite{Lusztig1989,Lusztig199001,Lusztig199002}, laid the foundations for the study of quantum groups at roots of unity. Furthermore, it is well known that there is a deep connection between the representation theory of algebraic groups over fields of characteristic $p$ and that of the corresponding quantum groups at $p$-th roots of unity \cite{AJS94}. Nevertheless, explicit constructions of irreducible modules remain unknown in full generality.

The Gelfand--Tsetlin construction provides a systematic approach to the explicit study of representations. Originating from the work of Gelfand and Tsetlin \cite{GT1950}, this theory has been developed extensively in subsequent works \cite{DOF1994, Molev2006, Zhe2017}, among the others. These techniques have proven to be particularly useful in the study of Gelfand--Tsetlin modules for $\mathfrak{gl}_n$ \cite{FGR2016, FGR2017, FRZ2019, V2017, V2018, Za2017}.
The construction of Gelfand--Tsetlin basis is extend to the quantum group $\U_q(\mathfrak{gl}_n)$ for generic $q$ in the foundational work of Jimbo \cite{Jimbo1988},
with a proof later provided by Ueno-Shibukawa-Takebayashi \cite{UST1989, UST1990}.
Mazorchuk and Turowska \cite{MT2000} subsequently studied generic Gelfand--Tsetlin modules.
Futorny, Ramires and Zhang investigated the irreducibility of generic modules \cite{FRZ201801} and constructed singular Gelfand--Tsetlin modules
\cite{FRZ201802}.

In this work, we extend the classical Gelfand¨CTsetlin construction to quantum groups at roots of unity of odd order.
For every generic tableau, we construct an infinite dimensional module,
and different choices of submodules yield non-isomorphic quotients,
all of which have Gelfand--Tsetlin multiplicities bounded by $1$.
The special case recovers the construction of finite dimensional modules in \cite{AAC1995, AC1991, DT1997}.
For $1$-singular tableaux,
we construct
singular Gelfand--Tsetlin modules with Gelfand--Tsetlin multiplicities bounded by $2$.
Our methods also apply to quantum groups at roots of unity of even order and to  Lie algebras of characteristic $p$.
For instance, our approach allows us to recover all non-restricted representations from \cite{W2017} and further extends the known classification.

The paper is organized as follows.
In Section \ref{section:pre}, we recall some general facts about quantum groups.
In Section \ref{section name:GT basis}, with $q$ a root of unity,
we give explicit constructions of generic and $1$-singular modules.
These constructions lead to new realizations of certain finite-dimensional modules and enable us to determine their irreducibility.
As an application, in Section \ref{section:ex} we apply the Gelfand--Tsetlin bases
to the study of baby Verma modules for $\U_q(\mathfrak{gl}_3)$ and determine their composition series.
\section{Preliminaries}\label{section:pre}
This section is devoted to a brief review of the quantum group and conventions that will be used throughout the paper.
\subsection{Quantum enveloping algebras}
Let $P$ be the free $\mathbb{Z}$-lattice of rank $n$ with the canonical basis
$\{\epsilon_1, \ldots, \epsilon_n\}$, i.e. $P = \bigoplus_{i=1}^n \mathbb{Z}\epsilon_i$,
endowed with symmetric bilinear form $\langle \epsilon_i, \epsilon_j \rangle = \delta_{ij}$.
Let $\Pi = \{\alpha_j = \epsilon_j - \epsilon_{j+1} \mid j = 1,2,\ldots,n-1\}$
and $\Phi = \{\epsilon_i - \epsilon_j \mid 1 \leq i \neq j \leq n\}$.
Then $\Phi$ realizes the root system of type $A_{n-1}$ with $\Pi$ a basis of simple roots.

Given a Lie algebra $\mathfrak{g}$ we denote its quantum enveloping algebra by $\U_q(\mathfrak{g})$.
Throughout the paper we abbreviate $\U_q=\U_q(\mathfrak{gl}_n)$.
We define $\U_q$ as a unital associative complex algebra generated by $E_i, F_i, K_j, K_j^{-1},
i = 1, 2, \ldots, n - 1, j = 1, 2, \ldots, n$ subject to the relations:
	\begin{align*}
		K_i K_j = K_j K_i, \quad K_i K_i^{-1} = K_i^{-1} K_i = 1,\\
		K_i E_j K_i^{-1} = q^{\delta_{ij}-\delta_{i,j+1}} E_j,\\
		K_i F_j K_i^{-1} =  q^{\delta_{i,j+1}-\delta_{ij}} F_j,\\
		[E_i, F_j] = \delta_{ij} \frac{K_i  K_{i+1}^{-1}-K_{i+1}  K_{i }^{-1}}{q - q^{-1}},\\
		[E_i, E_j] = [F_i, F_j] = 0, \quad |i - j| \geq 2,\\
		E_i^2 E_{i\pm 1} - (q + q^{-1}) E_i E_{i\pm 1} E_i + E_{i\pm 1} E_i^2 = 0,\\
		F_i^2 F_{i\pm 1} - (q + q^{-1}) F_i F_{i\pm 1} F_i + F_{i\pm 1} F_i^2 = 0.
	\end{align*}
When $q$ is not a root of unity,
the center of $\U_q(\mathfrak{gl}_n)$ is generated by the following $n+1$ elements \cite{FRT1990}:
\begin{align}\label{cnk}
c_{nk} = \sum_{\sigma,\sigma' \in S_n} (-q)^{-l(\sigma)-l(\sigma')} l_{\sigma(1),\sigma'(1)}^+ \cdots l_{\sigma(k),\sigma'(k)}^+ l_{\sigma(k+1),\sigma'(k+1)}^- \cdots l_{\sigma(n),\sigma'(n)}^-,
\end{align}
where $S_n$ is the $n$-th symmetric group, $0 \leq k \leq n$,
\begin{align*}
	l_{ii}^{\pm}=K_i^{\pm1},\
	l_{i,i+1}^{+}=(q-q^{-1})K_iF_i,\
	l_{i+1,i}^{-}=-(q-q^{-1})E_{i}K_i^{-1}.
\end{align*}
Denote by $\mathfrak{z}_q$ the center of the algebra $\U_q(\mathfrak{gl}_n)$.
When $q$ is a root of unity, $c_{nk}$ belong to the center of $\U_q(\mathfrak{gl}_n)$,
but they do not generate the center.

\begin{proposition}\cite{KS1997}
Let $q$ be a $m$-th root of unity, $\alpha$ be a positive root.
Then the elements $E_\alpha^m$, $F_\alpha^m$, and $K_i^m$ ($i=1,2,\dots,n$) belong to $\mathfrak{z}_q$.
\end{proposition}

Denote by $\mathfrak{z}_0$ the subalgebra of $\mathfrak{z}_q$ generated by the elements $E_\alpha^m$, $F_\alpha^m$, and $K_i^m$.
\subsection{Finite dimensional modules for $\U_q(\mathfrak{gl}_n)$}
We recall the Gelfand--Tsetlin bases for finite-dimensional representations of $\U_q(\mathfrak{gl}_n)$ when $q$ is not a root of unity.

A \emph{Gelfand--Tsetlin tableau} of height $n$ is an triangular array
	\medskip
	\begin{center}
		
		\Stone{\mbox{ $v_{n1}$}}\Stone{\mbox{ $v_{n2}$}}\hspace{1cm} $\cdots$ \hspace{1cm} \Stone{\mbox{$v_{n,n-1}$}}\Stone{\mbox{ $v_{nn}$}}\\[0.2pt]
		\Stone{\mbox{$v_{n-1,1}$}}\hspace{1.7cm} $\cdots$ \hspace{1.8cm} \Stone{\mbox{ $v_{n-1,n-1}$}}\\[0.3cm]
		\hspace{0.2cm}$\cdots$ \hspace{0.8cm} $\cdots$ \hspace{0.8cm} $\cdots$\\[0.3cm]
		\Stone{\mbox{ $v_{21}$}}\Stone{\mbox{$v_{22}$}}\\[0.2pt]
		\Stone{\mbox{ $v_{11}$}}\\
		\medskip
	\end{center}
with entries
 $\{v_{ij} \mid 1 \leq j \leq i \leq n\}$ in $\mathbb{C}$.
A Gelfand--Tsetlin tableau is called \emph{standard} if
$v_{ki} - v_{k-1,i} \in \mathbb{Z}_{\geq 0}$
and $v_{k-1,i} - v_{k,i+1} \in \mathbb{Z}_{>0}$
for all $1 \leq i \leq k \leq n$.

For $1 \leq j \leq i \leq n$, $\delta^{ij} \in \mathbb{Z}^{\frac{n(n+1)}{2}}$ is defined by $(\delta^{ij})_{ij}=1$ and all other $(\delta^{ij})_{k\ell}$ are zero. The weights of $\U_q(\mathfrak{gl}_n)$ are written as $n$-tuples $\lambda=(\lambda_1, \dots, \lambda_n)$.
The following theorem describes the Gelfand--Tsetlin basis for simple
finite-dimensional
$\U_q(\mathfrak{gl}_n)$-modules with a given highest weight.

\begin{theorem}\label{Th:GT formulae}
\cite{Jimbo1988}
Let $q$ be generic, $L(\lambda)$ be the finite dimensional irreducible module over $\U_{q}$.
Then there exist a basis $T(v)$ of $L(\lambda)$
parameterized by all
standard  tableaux $v$ with fixed top row $v_{nj}=\lambda_j-j+1$.
Moreover, the action of the generators of $\U_{q}$ on $L(\lambda)$ is given by the \emph{Gelfand--Tsetlin formulae}${: }$
	\begin{equation}
    \label{GT-formulae}
		\begin{aligned}
		K_k(T(v))&=q^{a_k}T(v),\quad a_k=\sum_{i=1}^{k}v_{k,i}-\sum_{i=1}^{k-1}v_{k-1,i}+k,\ k=1,\ldots,n,\\
			E_{k}(T(v))&=-\sum_{j=1}^{k}
			\frac{\prod_{i=1}^{k+1} [v_{k,j}-v_{k+1,i}]_q}{\prod_{i\neq j} [v_{k,j}-v_{k,i}]_q}
			T(v+\delta^{kj}),\\
			F_{k}(T(v))&=\sum_{j=1}^{k}\frac{\prod_{i=1}^{k-1} [v_{k,j}-v_{k-1,i}]_q}{\prod_{i\neq j} [v_{k,j}-v_{k,i}]_q}T(v-\delta^{kj}).\\
		\end{aligned}
	\end{equation}
where $v_{ki}=\lambda_{ki}-i+1$, $v\pm\delta^{kj}$ is  obtained
from $v$ by adding $\pm 1$ to the $(k,j)$th entry of $v$,
and $T(v)$ is zero if it is not standard.
\end{theorem}

The next proposition gives the explicit action of the generators of $\Gamma_q$.
\begin{proposition}
 \cite{FRZ201801}The generator $c_{mk}$ of $\Gamma_q$ acts on $T(v)$ as multiplication by
\begin{equation}
\label{action of  c_{mk}}
	\gamma_{mk}(v)=[k]_{q}![m-k]_{q}!q^{mk-m^2+k}\sum_{\tau}
q^{\sum_{i=1}^{k}v_{m\tau(i)}-\sum_{i=k+1}^{m}v_{m\tau(i)}},
\end{equation}
where $1 \leq m \leq n$,
$0 \leq k \leq m$,
$[k]_q! = [1]_q[2]_q\cdots [k]_q$ and
$\tau\in S_m$ is a $(k,m-k)$-shuffle such that $\tau(1)<\cdots<\tau(k), \tau(k+1)<\cdots<\tau(m)$.
\end{proposition}
\subsection{Gelfand--Tsetlin modules}
Consider the following chain:
$$\U_q(\mathfrak{gl}_1) \subset \U_q(\mathfrak{gl}_2) \subset \cdots \subset \U_q(\mathfrak{gl}_n).$$
The subalgebra of $\U_q$ generated by the centers of
$\U_q(\mathfrak{gl}_m)
$, $1\leq m\leq n$,
is called the Gelfand--Tsetlin subalgebra of $\U_q$ and is denoted by $\Gamma_q$.

\begin{definition}
A finitely generated $\U_q$-module $M$ is called a Gelfand--Tsetlin module (with respect to
$\Gamma_q$) if
$$M=\bigoplus_{\eta \in  \Gamma_q^*}M (\eta),$$
where
$M(\eta) = \{ v \in M \mid (z-\eta(z))^{k}v=0  \text{ for some } k\geq 0 \text{and any }z\in \Gamma_q\}$,
$\eta$ is called a Gelfand--Tsetlin weight of $M$ if $M(\eta) \neq 0$.
\end{definition}
By Equation \eqref{action of c_{mk}},
each Gelfand--Tsetlin tableau determines a unique Gelfand--Tsetlin weight. Conversely, any Gelfand--Tsetlin weight  $\eta$ corresponds to a unique tableau,  up to permutation of the entries within each row.
\section {Gelfand--Tsetlin bases for representation of $\U_q(\mathfrak{gl}_n)$ at roots of unity}\label{section name:GT basis}
In this section, we assume that
$q=e^{\mathbf{i}\theta}$, $\mathbf{i}=\sqrt{-1}$, is a root of unity.
Then $\left\{\frac{2k\pi }{\theta} \mid k \in \mathbb{Z}\right\}$ is the set of all complex $x$ such that $q^x = 1$.
For convenience, we take $\theta=\frac{2\pi}{m}$ and $m$ is odd.
For any complex number $x$, we define
$$ [x]_q = \frac{q^x - q^{-x}}{q - q^{-1}}.$$
Then
\begin{equation*}
	[x]_q=0\Longleftrightarrow x  \in m\mathbb{Z}.
\end{equation*}
If $0\leq \lambda_1-\lambda_n < m-n+2$, then Theorem \ref{Th:GT formulae} holds for root of unity.

Recall that $\mathfrak{z}_q$ denotes the center of the algebra $\U_q(\mathfrak{gl}_n)$.
Let $\operatorname{Spec}(\mathfrak{z}_q)$ be the set of all algebra homomorphisms $\chi : \mathfrak{z}_q \to \mathbb{C}$, and
$\operatorname{Rep}(\U_q(\mathfrak{gl}_n))$ the set of all equivalence classes of irreducible representations of $\U_q(\mathfrak{gl}_n)$.

If $\rho$ is an irreducible representation of $\U_q(\mathfrak{gl}_n)$ on a vector space $V$,
then $\rho(z)$, $z \in \mathfrak{z}_q$, is a scalar operator on $V$, that is, there exists a complex number $\chi_{\rho}(z)$ such that
$$
 \rho(z) = \chi_ \rho (z)I.
$$
Obviously, $\chi_\rho : \mathfrak{z}_q \to \mathbb{C}$ is an algebra homomorphism,
called the \textit{central character} of $\rho$.
Clearly, equivalent representation have the same central character.
Therefore,
the assignment $\rho \mapsto \chi_\rho$ defines a map
$$
\Psi : \operatorname{Rep}(\U_q(\mathfrak{gl}_n)) \to \operatorname{Spec}(\mathfrak{z}_q).
$$
The mapping $\Psi$ is surjective, but not injective in general.

Every irreducible representation of the algebra $\U_{q}(\mathfrak {gl}_n)$ is finite-dimensional.

\begin{definition}\label{cyclic}
  An irreducible representation of $\U_{q}(\mathfrak {gl}_n)$ is called \emph{cyclic} if $E_i ^m$ and $F_i ^m$, $i = 1, 2, \dots, n-1$,
  are nonzero scalar operators.
\end{definition}
A cyclic representation is neither a highest nor a lowest weight representation.

Let $\U_q^+$, $\U_q^-$ and $\U_q^0$ be the subalgebras of $\U_{q}(\mathfrak {gl}_n)$ generated by the $E_i$,
the $F_i$, and the $K_i,K_i^{-1}$ ($i=1,\dots,n$) respectively.
The quantum group admits the triangular decomposition
$$\U_q(\mathfrak{gl}_n)=\U_q^-\otimes \U_q^0\otimes \U_q^+.$$

Recall that $\mathfrak{z}_0$ is the subalgebra of $\mathfrak{z}_q$ generated by the elements $E_\alpha^m$, $F_\alpha^m$, and $K_i^m$.
For each $\chi \in \operatorname{Maxspec}(\mathfrak{z}_0)$ such that $\chi(E_\alpha^m)=0$,
 let $\mathfrak{m}_\chi$ be the corresponding maximal ideal. Define the reduced algebra $\U_\chi \;:=\; \U_q(\mathfrak{gl}_n)\big/\mathfrak{m}_\chi \U_q(\mathfrak{gl}_n)$ and denote by $\U_\chi^{\ge0}$ the image of $\U_q^{\ge0}=\U_q^0\otimes \U_q^+$ in $\U_\chi$.
The \emph{baby Verma module} is defined to be the induced module
\begin{align}\label{babyvm}
M_{\chi}(\lambda):=\U_{\chi} \otimes_{\U_{\chi}^{\geq 0}} \mathbb{C}v_\lambda,
\end{align}
where $\mathbb{C}v_\lambda$ is the one-dimensional $\U_{\chi}^{\geq 0}$-module with $E_{\alpha} \cdot v_\lambda = 0$ and $K_{i} \cdot v_\lambda = q^{\lambda_{i}}v_\lambda$.
Any irreducible highest weight module is a quotient of $M_{\chi}(\lambda)$.
\subsection {Generic Gelfand--Tsetlin modules}
\label{subsec:generic-GT-modules}
For $q$ is not a  root of unity, the Geland-Tsetlin basis given in Theorem \ref{Th:GT formulae} are common eigenvectors for the Gelfand--Tsetlin subalgebra. This section is devoted to construct
Gelfand--Tsetlin modules containing Gelfand--Tsetlin weights associated with the generic tableaux when $q$ is a root of unity.

A Gelfand--Tsetlin tableau $T(v)$ is called \emph{generic} if
it satisfies the following defining conditions${: }$
		$$v_{ki}-v_{kj}\notin \frac{\mathbb{Z}}{2} \text{ for all }1\leq k\leq n-1 \text{ and } i\neq j.$$
For a generic tableau $v$, Arnaudon and Chakrabarti showed that the vector space spanned by $T(v+z)$,
where $z_{ij}\in \mathbb{Z}_m$ is a simple $\U_q(\mathfrak{gl}_n)$-module and the action of generators is given by Gelfand--Tsetlin formulae \cite{AC1991, KS1997}.
Schur's Lemma implies that
the actions of the central elements $E_k^m$ and $F_k^m$ are scalar multiplications.
For some choice of $v_{k,i}$,
the actions of $E_k^m$ and $F_k^m$
are nonzero.
In this case, the modules are cyclic.\cite{AC1991, KS1997}.

We will provide an alternative construction that generalizes the results of Arnaudon and Chakrabarti.
For a generic tableau $v$, we denote
by ${\mathcal B}(T(v))$  the set of all Gelfand--Tsetlin tableaux $T(R)$ of height $n$
satisfying $r_{nj}=v_{nj}$ and $r_{ij}-v_{ij}\in\mathbb{Z}$ for $1\leq j\leq i \leq n-1$.

Let $T(v)$ be a fixed tableau of height $n$.
\begin{enumerate}
 \item
$\mathcal{B}(T(v)) := \left\{T(v+z) \mid z \in \mathbb{Z}^{\frac{n(n+1)}{2}} \text{ with } z_{ni} = 0 \text{ for } 1 \leq i \leq n\right\}$.	
\item
$V(T(v)) = \operatorname{span}\mathcal{B}(T(v))$.
\end{enumerate}

\begin{remark}
\label{assume}
Let $\mathbb{C}$ be a one-dimensional representation
such that $K_i\cdot 1=-1$, $K_{i+1}\cdot 1=-1$, and $K_j\cdot 1=1$ for all $j\neq i,i+1$.
Then
$$
V\!\left(T\left(v + \frac{m}{2}\delta^{ij}\right)\right)
\simeq V(T(v))\otimes \mathbb{C}
$$
as $U_q$-modules.
Then we can assume that for any $a,b$ and $i,j$, either $v_{a,b}-v_{i,j}\in\mathbb{Z}$ or $v_{a,b}-v_{i,j}\notin\frac{\mathbb{Z}}{2}$.
\end{remark}

By symmetry of the
Gelfand--Tsetlin formulas \eqref{GT-formulae}, the representation obtained is invariant under permutations of the entries within any given row.
Hence, we only consider tableaux that satisfy the following condition£º
\begin{align}
\label{condition of tableaux}
\text{ If } \{j\mid v_{k,j}-v_{k-1,i}\in\mathbb{Z}\}\neq \emptyset,
\text{ then } \{j\mid v_{k,j}-v_{k-1,i}\in\mathbb{Z}\}=\{i\}
\text{ and } v_{k,i}=v_{k-1,i}.
\tag{*}
\end{align}
For $1 \leq k \leq n-1$, the $k$-th row contains at most one entry $v_{k,j}$ satisfying $v_{k,j} - v_{k-1,i} \in \mathbb{Z}$ and this condition implies  $v_{k,i} = v_{k-1,i}$.
For $k = n$, the $n$-th row can contain several entries $v_{n,j}$ with $v_{n,j} - v_{n-1,i} \in \mathbb{Z}$. Whenever this occurs, we have  $v_{n,i} = v_{n-1,i}$.

Let $\U_q(T(v))$ be the submodule of $V(T(v))$ generated by $T(v)$.
For any $(k,j)$ with $v_{kj} =v_{k-1,j} = \cdots =v_{tj}\neq v_{t-1,j}$,
denote by $W_{kj}^{(1)}$
the submodule of $\U_q(T(v))$ generated by
$$
T\bigl(v-m\delta^{kj}-m\delta^{k-1,j}-\cdots-m\delta^{tj}\bigr)-T(v).
$$

\noindent\textbf{Case {\upshape I}.}
If $v_{k+1,j} \neq v_{kj}$, then we set $W_{kj}=W_{kj}^{(1)}$.

\medskip
\noindent\textbf{Case {\upshape II}.}
If $v_{k+1,j}=v_{kj}$,
then
$$
W_{kj}^{(2)}:=\U_q\Bigl(T\bigl(v-m\delta^{kj}-m\delta^{k-1,j}-\cdots-m\delta^{tj}\bigr)\Bigr)
$$
is a proper submodule of $\U_q(T(v))$.
In this case we take $W_{kj}$ to be either $W_{kj}^{(1)}$ or $W_{kj}^{(2)}$.

Define
\begin{align}
\label{def of VT--generic}
\overline{V}(T(v)) = \U_q(T(v)) \bigg/ \sum_{k,j} W_{kj}.
\end{align}
It will be shown later (Theorem \ref{thm:structure}) that $\overline{V}(T(v))$ is finite dimensional.
We use the same notation $T(v+z)$ for vectors in this quotient module.

\begin{example}
Let $T(v)$ be a generic tableau with $v_{3,1}=v_{2,1}=v_{1,1}$ and $v_{2,2}-v_{k,j}\notin \tfrac{\mathbb{Z}}{2}$ for any $(k,j)\neq (2,2)$.
Define the corresponding submodules $W_{21}, W_{22}, W_{11}$ as follows$:$
$$
\begin{aligned}
W_{21} &:= \langle\, T(v - m\delta^{21} - m\delta^{11}) - T(v) \,\rangle
        \quad\text{or}\quad
        \langle\, T(v - m\delta^{21} - m\delta^{11}) \,\rangle,\\[4pt]
W_{22} &:= \langle\, T(v - m\delta^{22}) - T(v) \,\rangle,\\[4pt]
W_{11} &:= \langle\, T(v - m\delta^{11}) - T(v) \,\rangle
        \quad\text{or}\quad
        \langle\, T(v - m\delta^{11}) \,\rangle.
\end{aligned}
$$
The quotient module is then defined by
$$
\overline{V}(T(v))
:= \U_q(T(v))\big/ \bigl( W_{21} + W_{22}+W_{11} \bigr),
$$

\end{example}

\begin{theorem}\label{thm:structure}
Let $T(v)$ be a generic Gelfand--Tsetlin tableau.
\begin{enumerate}[{\rm(1)}]
\item
$\overline{V}(T(v))$ is a $\U_{q}$-module with dimension
$m^{\frac{n(n-1)}{2}}$,
and the action of generators is given by Gelfand--Tsetlin formulae \eqref{GT-formulae}.
\item
The action of generators $c_{rs}$ of Gelfand--Tsetlin subalgebra on  $\overline{V}(T(v))$ is given by
$$c_{rs}T(w)=\gamma_{rs}(w)T(w)$$
for any $T(w)=T(v+z)\in \overline{V}(T(v)) $
where $\gamma_{rs}$ is defined in
\eqref{action of c_{mk}}.
\item
For any fixed $v$ and distinct choices of $W_{kj}$ with $1\leq j\leq k\leq n-1$,
the quotients $\overline{V}(T(v))$
are pairwise non-isomorphic.
\end{enumerate}
\end{theorem}
\begin{proof}
   A basis for $\overline{V}(T(v))$
   is given by the vectors $T(v+z)$, where $z=(z_{kj})$ satisfies
\begin{equation}
\label{condition}
\begin{aligned}
0\leq z_{k+1,j}-z_{k,j} \leq m-1,\\
1\leq j\leq k\leq n-1,
\end{aligned}
\end{equation}
here for any $(k,j)$,
$z_{kj}\in\mathbb{Z}_m$ if $W_{kj}=W_{kj}^{(1)}$,
 and $z_{kj}\in\mathbb{Z}$ if $W_{kj}=W_{kj}^{(2)}$.
For any $z_{k,j}\in \mathbb{Z}_m$ appearing in the inequalities,
we take the representatives from the set $\{-(m-1),\dots,0\}$.
Then the dimension of this module is
$m^{\frac{n(n-1)}{2}}$.
The rest of (1) and (2) are straightforward.

We now prove (3).
Recall that $W_{kj}^{(1)}$ is the submodule of $\U_q(T(v))$ generated by
$$T(v-m\delta^{kj}-m\delta^{k-1,j}-\cdots -m\delta^{tj})-T(v),$$
and $W_{kj}^{(2)}$ is the submodule of $\U_q(T(v))$ generated by
$$T(v-m\delta^{kj}-m\delta^{k-1,j}-\cdots -m\delta^{tj}).$$
For all $(r,s)\neq(k,j)$ the submodules $W_{rs}$ and $W_{rs}'$ are chosen arbitrarily,
we then set
$$
M^{(1)}:=\U_q(T(v))\Big/\!\Bigl(W_{kj}^{(1)}+\sum_{(r,s)}W_{rs}\Bigr),\qquad
M^{(2)}:=\U_q(T(v))\Big/\!\Bigl(W_{kj}^{(2)}+\sum_{(r,s)}W'_{rs}\Bigr).
$$

Suppose that there is an isomorphism $\varphi:M^{(1)}\to M^{(2)}$,
then the image of any
Gelfand--Tsetlin weight vector in $M^{(1)}$ is a Gelfand--Tsetlin weight vector with the same Gelfand--Tsetlin weight.
Now choose a basis vector $T(v+z)$ with
$z_{k+1,j}=0, z_{k,j}=-(m-1),$ and $z_{k,j}\neq z_{k-1,j}.$
Since the basis vectors in $M^{(1)}$ (resp. $M^{(2)}$) have the distinct Gelfand--Tsetlin weights,
we have $\varphi(T(v+z))=\alpha_z T(v+z)$,
where $\alpha_z$ is a non-zero scalar depending on $z$.
Since $\varphi(T(v+z))=\alpha_z T(v+z)$ with $\alpha_z\neq0$,
applying $F_k$ on both sides gives
$$
\varphi\bigl(F_kT(v+z)\bigr)
 =\alpha_z\,F_k T(v+z).
$$
By definition,
$T(v+z-\delta^{kj})$ is nonzero
in $M^{(1)}$ and
$T(v+z-\delta^{kj})$ is zero
in $M^{(2)}$.
The coefficient of $T(v+z-\delta^{kj})$ in $F_k T(v+z)\in M^{(1)}$ is nonzero.
We arrive at a contradiction.
\end{proof}

\begin{remark}
If $W_{kj} = W_{kj}^{(1)}$ for all $(k,j)$,
then $\overline{V}(T(v))$ coincides with the modules in \cite{AAC1995, AC1991}.
Suppose that $v$ is a generic tableau such that $v_{k+1,j} = v_{k,j}$ for all $1\leq j\leq k\leq n-1$.
Then by taking distinct choices of $W_{kj}$, we obtain
$2^{\frac{n(n-1)}{2}}$
pairwise non-isomorphic modules.
\end{remark}

By Theorem \ref{thm:structure}, we have the following corollary.
\begin{corollary}
\label{generic:separates distinct tableaux}
The basis vectors $T(v+z)$
of $\overline{V}(T(v))$ have distinct Gelfand--Tsetlin weights.
Let $\sum_z c_zT(v+z)$ be an element of $\overline{V}(T(v))$ with $c_{z_0}\neq0$,
then there exists an element $\gamma \in \Gamma_q$ such that $\gamma\cdot \sum_zc_zT(v+z)=T(v+z_0).$
\end{corollary}

It follows from Corollary~\ref{generic:separates distinct tableaux} that every submodule of $\overline{V}(T(v))$ is a Gelfand--Tsetlin weight module with a basis consisting of a subset of the Gelfand--Tsetlin basis. Hence, we have the following result.

\begin{corollary}
\label{cor:maximal-submodule}
Each module $\overline{V}(T(v))$ contains a unique maximal submodule.
\end{corollary}

The following theorem gives a necessary and sufficient condition for $\overline{V}(T(v))$ to be a simple module.

\begin{theorem}\label{thm:irreducibility}
The module $\overline{V}(T(v))$ is simple if
and only if the following holds:
for any $i\in\{1,\dots,n\}$ and $j\in\{1,\dots,n-1\}$ with
$W_{n-1,j}=W_{n-1,j}^{(2)}$, if $v_{n,i}-v_{n-1,j}\in \mathbb{Z}$,
then $v_{n,i}-v_{n-1,j}=0  \bmod m$.
\end{theorem}
\begin{proof}
Recall that the basis of $\overline{V}(T(v))$ is given in the proof of Theorem \ref{thm:structure}.
If $W_{n-1,j}=W_{n-1,j}^{(2)}$, then $v_{n,j}=v_{n-1,j}$ and
any basis vector $T(v+z)$
satisfies
\begin{equation}\label{extra-condition}
\begin{aligned}
-(m-1) \leq z_{n-1,j} \leq 0.
\end{aligned}
\end{equation}
Suppose that $v_{n,i}\neq v_{n,j} \bmod~m$ and
$v_{n,i}-v_{n-1,j}\in \mathbb{Z}$ for some $i$.
In view of (\ref{extra-condition}),
there exists a basis vector $T(v+z')$ with $z'_{n-1,j}=a$, $-(m-1)\leq a<0$ and $v_{n-1,j}+a=v_{n,i}\bmod m$.
For any $T(v+z'')\in \U_q(T(v+z'))$, we have
\begin{align*}
-(m-1) \leq z''_{n-1,j} \leq a<0.
\end{align*}
Thus the submodule generated by $T(v+z')$ is a proper submodule of $\overline{V}(T(v))$,
which implies that $\overline{V}(T(v))$ is not simple.

Let $0\neq \xi = \sum a_z T(v+z)\in \overline{V}(T(v))$
with non-zero $a_z\in \mathbb{C}$.
It follows from Corollary \ref{generic:separates distinct tableaux} that
there exists some $T(v+z^{(n)}) \in \U_q(\xi)$.
It suffices to show that $T(v)\in \U_q(T(v+z^{(n)}))$.

For any fixed $j$, assume that $W_{n-1,j}=W_{n-1,j}^{(1)}$ and
\[
v_{n-1,j}+z_{n-1,j}=v_{n-2,j}+z_{n-2,j}=\cdots= v_{s,j}+z_{s,j}\neq v_{s-1,j}+z_{s-1,j}.
\]
By the Gelfand--Tsetlin formulas \eqref{GT-formulae},
the coefficient of $T(v+z^{(n)}-\delta_{sj})$ in $F_s(T(v+z^{(n)}))$ is nonzero.
It follows that $T(v+z^{(n)}-\delta_{sj})\in \U_q(T(v+z^{(n)}))$.
Using the same argument above and considering the action of $F_{s+1},\cdots, F_{n-1}$
successively,
we conclude that
\[
T(v+z^{(n)}-\delta_{s,j}-\delta_{s+1,j}-\cdots-\delta_{n-1,j})\in \U_q(T(v+z^{(n)})).
\]
After finitely many steps we obtain some $T(v+z)\in \U_q(T(v+z^{(n)}))$ with $z_{n-1,j}=0$ and $z_{n-1,i}=z^{(n)}_{n-1,i}$ for $i\neq j$.

For any fixed $j$,
assume that $W_{n-1,j}=W_{n-1,j}^{(2)}$.
In this case $v_{n,j}=v_{n-1,j}$ and
if $v_{n,i}-v_{n-1,j}\in \mathbb{Z}$,
then $v_{n,i}=v_{n-1,j} \bmod m$.
Moreover, if $z^{(n)}_{n-1,j}\neq 0$,
then the coefficient of $T(v+z^{(n)}+\delta_{n-1,j})$ in $E_{n-1}(T(v+z^{(n)}))$ is nonzero.
It follows that $T(v+z^{(n)}+\delta_{n-1,j})\in \U_q(T(v+z^{(n)}))$.
After finitely many steps we obtain some $T(v+z)\in \U_q(T(v+z^{(n)}))$ with $z_{n-1,j}=0$ and $z_{n-1,i}=z^{(n)}_{n-1,i}$ for $i\neq j$.

The above observation implies that there exists some $T(v+z^{(n-1)})\in \U_q(T(v+z^{(n)}))$ such that
$z^{(n-1)}_{n-1,j}=0$ for any $1\leq j\leq n-1$.
Iterating the above process,
we obtain that $T(v)\in \U_q(T(v+z^{(n)}))$, as required.
\end{proof}
If $v_{n,i}-v_{n-1,j}\in\mathbb{Z}$ with $i\neq j$, then
we take $v_{n,i}$ as a formal variable.
By Theorem \ref{thm:irreducibility}, $\overline{V}(T(v))$ is simple.
Schur's Lemma implies that
the central elements $E_k^m$ and $F_k^m$ act as scalars on $\overline{V}(T(v))$.
Moreover,
the basis of $\overline{V}(T(v))$ is independent choice of these $v_{n,i}$.
Hence $E_k^m$ and $F_k^m$ act as scalars on $\overline{V}(T(v))$ for any $v_{n,i}\in\mathbb{C}$.
This leads the following result.
\begin{proposition}\label{prop:cyclic-module-criterion}
The actions of $E_k^m$ and $F_k^m$ on $T(v)$ are given as follows:
\begin{align*}
E_k^m T(v) &= -\sum_j\prod_{r=0}^{m-1}
             \frac{\displaystyle
             \prod\limits_{i=1}^{k+1} [v_{k,j} - v_{k+1,i} + r]_q}
             {\displaystyle\prod_{i \neq j}
             [v_{k,j} - v_{k,i}+r]_q} \; T(v),\\[4mm]
F_k^m T(v) &= \sum_j\prod_{r=0}^{m-1}
             \frac{\displaystyle
             \prod\limits_{i=1}^{k-1} [v_{k,j} - v_{k-1,i} - r]_q}
             {\displaystyle\prod_{i \neq j}
             [v_{k,j} - v_{k,i}-r]_q} \; T(v).
\end{align*}
The sum is taken over all $1\leq j\leq k$ such that $W_{kj}=W_{kj}^{(1)}$, and
$\overline{V}(T(v))$ is cyclic if and only if the coefficients are nonzero for all $k$.
\end{proposition}

For any fixed $v$,
by Theorem \ref{thm:structure}
,
we can construct non-isomorphic cyclic modules with the same Gelfand--Tsetlin weights.
If $W_{kj} = W_{kj}^{(1)}$ for all $(k,j)$,
then the cyclic module $\overline{V}(T(v))$ coincides with the cyclic module constructed in \cite{AAC1995, AC1991}.

\begin{remark}
In the special case $v_{k+1,j}=v_{k,j}$ for all $1\leq j\leq k\leq n-1$,
$$
\overline{V}(T(v)) = U_q(T(v)) \bigg/ \sum_{k,j} W_{kj}
$$
coincides with a certain baby Verma module $M_{\chi}(\lambda)$ \eqref{babyvm}.
Theorem~\ref{thm:irreducibility} provides a necessary and sufficient condition for such baby Verma module to be irreducible.
\end{remark}
\subsection {1-singular Gelfand--Tsetlin modules}
\label{subsec:1-singular-GT-modules}
In Section~\ref{subsec:generic-GT-modules},
every module we constructed is a Gelfand--Tsetlin weight module and the weight spaces are $1$-dimensional.
The Gelfand--Tsetlin weights correspond to generic tableaux.
It is natural to consider
the Gelfand--Tsetlin module containing the
Gelfand--Tsetlin weight associated with a singular tableau.

A  tableau $v$ is called
\emph{singular}
there exist $1\leq s<t\leq r\leq n-1$ such that
$v_{rs}-v_{rt}\in \frac{\mathbb{Z}}{2}  $. A singular tableau $v$ is called
\emph{1-singular} if there exist $i,j,k$ with $1 \leq i < j \leq k \leq n-1$ such that
$$v_{ki}-v_{kj}\in \frac{\mathbb{Z}}{2}  $$
and $$v_{rs} - v_{rt} \notin \frac{\mathbb{Z}}{2}  \text{ for all } (r,s,t) \neq (k,i,j).$$

Now we introduce the standard tableaux $T(v+z)$ and derivative tableaux
$\mathcal{D}T(v+z)$ for every $z=(z_{ij})$, where
\(z_{ki},z_{kj}\in \mathbb Z_m, z_{rs}\in \mathbb Z \text{ for } (r,s)\notin\{ (k,i),(k,j)\} \).
We impose the conditions $T(v+z) = T(v+\tau(z))$ and
$
\mathcal{D}T(v+z) + \mathcal{D}T(v+\tau(z)) = 0,
$
where $\tau(z)$ is obtained from $z$ by switching $z_{ki}$ and $z_{kj}$ and all other entries are fixed.

Similar to Remark~\ref{assume}, we can assume that for any $a,b$ and $i,j$, either $v_{a,b}-v_{i,j}\in\mathbb{Z}$ or $v_{a,b}-v_{i,j}\notin\frac{\mathbb{Z}}{2}$.
We set $V(T(v))$ to be the vector space spanned by the set of tableaux
$$\{T(v+z), \mathcal{D}T(v+z)  ,z_{ki},z_{kj}\in \mathbb Z_m, z_{rs}\in \mathbb Z \text{ for } (r,s)\notin\{ (k,i),(k,j)\} \}.$$

Replace $v$ by $v(x,y)$ where $v(x,y)_{rs} = v_{rs}$ if $(r,s) \notin \{(k,i),(k,j)\}$,
and $v(x,y)_{ki} = x$, $v(x,y)_{kj} = y$.
Then $$V(T(v(x,y)))=span_{\mathbb{C}(x,y)}\{T(v(x,y)+z)|z_{ki},z_{kj}\in \mathbb Z_m, z_{rs}\in \mathbb Z, (r,s) \notin \{(k,i),(k,j)\} \}$$
is a $\U_q$-module with action of the generators given by the formulas \eqref{GT-formulae}.

We denote by $\mathcal{H}$ the hyperplane $x-y =0$ and
$\mathcal{F}$ the subfield of $\mathbb{C}(x,y)$ consisting of all rational functions that are smooth on ${\mathcal{H}}$.
Recall that $q=e^{\mathbf{i}\theta}$ with $\mathbf{i}=\sqrt{-1}$.
Define linear map $\mathcal{D}^{v} \colon \mathcal{F} \otimes V(T(v(x,y))) \to V(T(v))$ by
$$
\mathcal{D}^{v}(fT(v(x,y)+z)) = \mathcal{D}^{v}(f)T(v+z) + f(v)\mathcal{D}T(v+z),
$$
where
$$\mathcal{D}^{v}(f) = \frac{q-q^{-1}}{4\mathbf{i}\theta }\left(\frac{\partial f}{\partial x} - \frac{\partial f}{\partial y}\right)\bigg|_{x=y=0}.
$$
In particular,
$$
\begin{aligned}
	\mathcal{D}^{v}([x-y]_{q}T(v(x,y)+z)) = T(v+z),\quad
	\mathcal{D}^{v}(T(v(x,y)+z)) = \mathcal{D}T(v+z).
\end{aligned}
$$	

The following Theorem shows that $V(T(v))$ has a $\U_q$-module structure.
\begin{theorem}\label{1-singular(1)}
If $v$ is an $1$-singular tableau, then $V(T(v))$ is a $\U_q$-module, the action of the generators of $u\in \U_{q}$ given by
\begin{equation}\label{TvDTv action}
\begin{aligned}
&u(T(v+z))  = \mathcal{D}^{v}([x-y]_{q} u(T(v+z))),  \\
& u(\mathcal{D} T({v}+z'))  = \mathcal{D}^{v}(u(T(v+z'))), \;z'\neq \tau(z).
\end{aligned}
\end{equation}

Moreover, for any $1\leq r\leq s \leq n$, the action of generators of
Gelfand--Tsetlin subalgebra on the basis vectors is given by
\begin{align*}
&c_{rs}(T(v+z)) = \gamma_{rs}(v+z)T((v+z)),\\
&c_{rs}(D(T(v+z))) =\gamma_{rs}(v+z)DT(v+z)+D^v\left(\gamma_{rs}(v+z)\right)T((v+z)).
\end{align*}
\end{theorem}

\begin{proof}
Recall that $V(T(v(x,y)))=span_{\mathbb{C}(x,y)}\{T(v(x,y)+z)\}$
is a $\U_q$-module with action of the generators given by the formulas (\ref{GT-formulae}),
and it is spanned by the following elements over $\mathbb{C}(x,y)$:
\[
S(z):= \frac{T(v(x,y)+z) + T(v(x,y)+\tau(z))}{2},
\]
\[
A(z):= \frac{T(v(x,y)+z) - T(v(x,y)+\tau(z))}{2[x-y]_q}.
\]
Notice  that $S(\tau(z)) = S(z)$ and $A(\tau(z)) = -A(z)$,
in particular if $z =\tau(z)$ then $S(z) = T(z)$ and $A(z) = 0$.
Let $L \subset V(T(v(x,y)))$ be the $\mathcal{F}$-span of $\{S(z), A(z)\}$.
We consider the action of $\U_q(\mathfrak{gl}_n)$ on $L$.

Let $I=\{(r,s)\,\big|\, 1\leq s\leq r\leq n\}$, and let $\tau\colon I\to I$ be the involution that interchanges $(k,i)$ and $(k,j)$, while leaving all other elements of $I$ fixed.
For a function $f$ by
$\tau\cdot f$ we denote the function
$\tau\cdot f(x,y)= f(y,x)$.
By (\ref{GT-formulae}) we have
\begin{align*}
E_{r}(T(v(x,y)+z))&=-\sum_{s=1}^{r}
\frac{\prod_{t=1}^{r+1} [v(x,y)_{r,s}-v(x,y)_{r+1,t}]_q}{\prod_{t\neq s} [v(x,y)_{r,s}-v(x,y)_{r,t}]_q}
T((x,y)+z+\delta^{rs}).
\end{align*}
Set
\[
f_{r,s}(x,y):= \frac{\prod_{t=1}^{r+1} [v(x,y)_{r,s}-v(x,y)_{r+1,t}]_q}{\prod_{t\neq s} [v(x,y)_{r,s}-v(x,y)_{r,t}]_q},
\]
then $\tau\cdot f_{r,s}=f_{\tau(r,s)}$.
By definition we get
\begin{align*}
E_rS(z)=&
      -\frac{1}{2}\sum_{s=1}^{r}  \left(f_{r,s}(z)+f_{\tau(r,s)}(\tau(z))
        \right)
        S(z+\delta^{rs})\\
        &\quad +\left(f_{r,s}(z)-f_{\tau(r,s)}(\tau(z))
        \right)
        [x-y]_qA(z+\delta^{rs}).
\end{align*}
Similarly,
\begin{align*}
E_rA(z)=&
       -\frac{1}{2}\sum_{r=1}^{s}
        \left(f_{r,s}(z)-f_{\tau(r,s)}(\tau(z))
        \right)
        \frac{1}{[x-y]_q} S(z+\delta^{rs})\\
        &\qquad\qquad +
        \left(f_{r,s}(z)+f_{\tau(r,s)}(\tau(z))
        \right)
       A(z+\delta^{rs}).
\end{align*}
It follows that both $E_rS(z)$ and $E_rA(z)$ belong to $L$.
By the same token, $L$ is invariant under the action of $F_r$ and $K_r$.
As a result, $\U_q(\mathfrak{gl}_n).L\subseteq L$.
Hence $L$ is a $(\U_q(\mathfrak{gl}_n),\mathcal{F})$-bimodule.

Take $\mathbb{C}$ to be the $1$-dimensional $\mathcal{F}$-module such that $f(x,y) \cdot 1\mapsto f(0,0)$.
Then $L \otimes_{\mathcal{F}} \mathbb{C}$ is a $\U_q(\mathfrak{gl}_n)$-module.
The assignment $S(z) \otimes 1\mapsto T(v+z)$, $A(z) \otimes 1\mapsto DT(v+z)$ defines an isomorphism from $L \otimes_{\mathcal{F}} \mathbb{C}$ to $V(T(v))$.
Hence $V(T(v))$ has $\U_q(\mathfrak{gl}_n)$-module structure via this isomorphism.

We now give the action of generators,
\begin{align*}
E_r(S(z)\otimes 1)=&
      -\frac{1}{2}\sum_{s=1}^{r}  \left(f_{r,s}(z)+f_{\tau(r,s)}(\tau(z))
        \right)
        S(z+\delta^{rs})
        \otimes 1\\
        &\quad +\left(f_{r,s}(z)-f_{\tau(r,s)}(\tau(z))
        \right)
        [x-y]_qA(z+\delta^{rs})
        \otimes 1.
\end{align*}
Similarly, we obtain the action of all generators on $S(z)\otimes 1$ and $A(z)\otimes 1$.
Then the action of generators on $V(T(v))$ follows.

The action of $c_{rs}$ is given by
\begin{align*}
c_{rs}S(z) &= \frac{1}{2}\left(\gamma_{rs}(v(x,y)+z)+\gamma_{rs}(v(x,y)+\tau(z))\right)
S(z)\\
&+\frac{1}{2}\left(\gamma_{rs}(v(x,y)+z)-\gamma_{rs}(v(x,y)+\tau(z))\right)
[x-y]_qA(z),
\end{align*}
and
\begin{align*}
c_{rs}A(z) &= \frac{1}{2}\left(\gamma_{rs}(v(x,y)+z)-\gamma_{rs}(v(x,y)+\tau(z))\right)
\frac{1}{[x-y]_q}S(z)\\
&+\frac{1}{2}\left(\gamma_{rs}(v(x,y)+z)+\gamma_{rs}(v(x,y)+\tau(z))\right)
A(z).
\end{align*}
Therefore, we have that
\begin{align*}
&c_{rs}(T(v+z)) = \gamma_{rs}(v+z)T((v+z)),\\
&c_{rs}(\mathcal{D}(T(v+z))) =\gamma_{rs}(v+z)\mathcal{D}T(v+z)+\mathcal{D}^{v}\left(\gamma_{rs}(v+z)\right)T((v+z)).
\end{align*}
This completes the proof.
\end{proof}

The 1-singular modules $V(T(v))$ are
infinite-dimensional.
Analogously to the generic case, we obtain finite-dimensional subquotients $\overline{V}(T(v))$ by taking appropriate quotient modules.

For any $(r,s)\notin \{(k,i),(k,j)\}$ with $1\leq s\leq r\leq n-1$ and $v_{rs} = v_{r-1,s} = \cdots = v_{ts} \neq v_{t-1,s}$, let $W_{rs}^{(1)}$ be the submodule of $\U_q(T(v))$ generated by
$$T\bigl(v-m\delta^{rs}-m\delta^{r-1,s}-\cdots-m\delta^{ts}\bigr)-T(v).$$

\noindent\textbf{Case {\upshape I}.}
If $v_{r+1,s} \neq v_{rs}$,
then we set $W_{rs}=W_{rs}^{(1)}$.

\noindent\textbf{Case {\upshape II}.}
If $v_{r+1,s} = v_{rs}$,
let $W_{rs}^{(2)}$ be the submodule of $\U_q(T(v))$ generated by
$$T\bigl(v-m\delta^{rs}-m\delta^{r-1,s}-\cdots-m\delta^{ts}\bigr).$$
We then take $W_{rs}$ to be either $W_{rs}^{(1)}$ or $W_{rs}^{(2)}$.

Define the quotient module
\begin{align}\label{def of VT--1sig}
\overline{V}(T(v)) = \U_q(T(v)) \bigg/ \sum_{r,s} W_{rs}.
\end{align}

To simplify subsequent discussion, we set up the following notation:
\begin{align*}
&L(z): =\begin{cases}
T(v+z), & \text{if } z_{ki}\geq z_{kj} \\
\mathcal{D}T(v+z). & \text{if } z_{ki}<z_{kj}
\end{cases}\\
\end{align*}
The set $\{L(z)\}$
forms a basis of $V(T(v))$.
Note that for any $z_{kj}\in \mathbb{Z}_m$ appearing in the inequalities,
we take the representatives from the set $\{-(m-1),\dots,0\}$.

A basis for $\overline{V}(T(v))$
   is given by the vectors $L(z)$, where $z=(z_{kj})$ satisfies
\begin{equation}
\label{condition-1}
\begin{aligned}
z_{ki},z_{kj}\in\mathbb{Z}_m\\
0\leq z_{r+1,s}-z_{r,s} \leq m-1,
\end{aligned}
\end{equation}
here for any $(r,s)\notin \{(k,i),(k,j)\}$,
$z_{rs}\in\mathbb{Z}_m$ if $W_{rs}=W_{rs}^{(1)}$,
 and $z_{rs}\in\mathbb{Z}$ if $W_{rs}=W_{rs}^{(2)}$.
For any $z_{r,s}\in \mathbb{Z}_m$ appearing in the inequalities,
we take the representatives from the set $\{-(m-1),\dots,0\}$.
The dimension of this module is
$m^{\frac{n(n-1)}{2}}$.

For example,
we can take some $v$ with $v_{ki} - v_{kj}\in\mathbb{Z}$ such that
$\overline{V}(T(v))$ is cyclic.
Its Gelfand--Tsetlin weight is different from those of the modules given in \eqref{def of VT--generic}.

By Theorem \ref{1-singular(1)}, we have the following corollary.
\begin{corollary}
\label{cor:1singular-separates}
Let $\sum_j c_jL(z_j)$ be an element of $\overline{V}(T(v))$ with $c_{z_0}\neq0$.
Then there exists an element $\gamma \in \Gamma_q$ such that $\gamma\cdot \sum_jc_jL(z_j)=L(z_0).$
\end{corollary}
It follows from Corollary~\ref{cor:1singular-separates} that every submodule of $\overline{V}(T(v))$ is spanned by a subset of the Gelfand--Tsetlin basis. Consequently, each $\overline{V}(T(v))$ contains a unique maximal submodule.

\begin{theorem}
\label{thm:pairwise-non-iso-modules}
For any fixed $v$ and distinct choices of $W_{rs}$ with $1\leq s\leq r\leq n-1$,
the modules $\overline{V}(T(v))$
are pairwise non-isomorphic.
\end{theorem}
\begin{proof}
Fix $(r,s)\notin \{(k,i),(k,j)\}$ with $1\leq s\leq r\leq n-1$.
For every $(a,b)\notin \{(r,s),(k,i),(k,j)\}$,
we choose $W_{ab}$ and $W'_{ab}$ arbitrarily and set
$$
N^{(1)}:=\U_q(T(v))\Big/\!\Bigl(W_{rs}^{(1)}+\sum_{(a,b)}W_{ab}\Bigr),\qquad
N^{(2)}:=\U_q(T(v))\Big/\!\Bigl(W_{rs}^{(2)}+\sum_{(a,b)}W'_{ab}\Bigr).
$$

We now show that $N^{(1)}$ and $N^{(2)}$ are not isomorphic.
Suppose that there exists an isomorphism $\phi:N^{(1)}\to N^{(2)}$,
then the image of any
Gelfand--Tsetlin weight vector in $N^{(1)}$ is a
Gelfand--Tsetlin weight vector with the same Gelfand--Tsetlin weight.
Now choose the basis vector $T(v+z), \mathcal{D}T(v+z)\in N^{(1)}$ with
$z_{k+1,j}=0, z_{k,j}=-(m-1),$ and $z_{k,j}\neq z_{k-1,j}.$
By Corollary~\ref{cor:1singular-separates},
we have $\phi(T(v+z))=a_z\, T(v+z)$,
$\phi(\mathcal{D}T(v+z))=b_z\, \mathcal{D}T(v+z)$,
where $a_z$ and $b_z$ are nonzero scalars depending on $z$.
Applying $F_k$ on both sides gives
\begin{align*}
\phi\bigl(F_kT(v+z)\bigr)
=a_z\,F_k T(v+z),\\
\phi\bigl(F_k \mathcal{D}T(v+z)\bigr)
=b_z\,F_k \mathcal{D}T(v+z).
\end{align*}
The rest of the proof follows the arguments of Theorem~\ref{thm:structure}(3) .
\end{proof}

We provide a sufficient condition for the module $\overline{V}(T(v))$ to be irreducible.
\begin{theorem}
\label{thm:irreducibility-1}
If $W_{rs}=W_{rs}^{(1)}$
for any $(r,s)\notin \{(k,i),(k,j)\}$ with $1\leq s\leq r\leq n-1$,
then
$\overline{V}(T(v))$
is simple.
\end{theorem}
\begin{proof}
By Corollary \ref{cor:1singular-separates},
it suffices to show the following statements:\\
(1) For any fixed $T(v+z)$, we have $T(v+z')\in \U_q T(v+z)$ for any $z'$,\\
(2) For any fixed $\mathcal{D}T(v+z)$, we obtain $\mathcal{D}T(v+z') \in \U_q\mathcal{D}T(v+z)$ for any $z'$.\\
(3) There exists  $T(v+z)$ such that $\mathcal{D}T(v+z') \in \U_qT(v+z)$ for some $z'$,\\
(4) There exists $\mathcal{D}T(v+z)$ such that $T(v+z') \in \U_q\mathcal{D}T(v+z)$ for some $z'$.

The proof of (1) and (2) can be proven by same argument as in the proof of Theorem \ref{thm:irreducibility}.
For (3),
we take $z$ with $z_{ki}=z_{kj}$ and $z_{k-1,a}\neq z_{ki}$ for any $a$.
Then the coefficient of $\mathcal{D}T(v+z-\delta^{ks})$ in $F_k T(v+z)$ is nonzero.
According to Corollary \ref{cor:1singular-separates},
$\mathcal{D}T(v+z-\delta^{ks}) \in \U_qT(v+z)$.
To get (4),
we take $z$ with $z_{ki}\neq z_{kj}$ and $z_{k-1,a}\notin\{z_{ki},z_{kj}\}$ for any $a$.
The coefficient of $T(v+z-\delta^{ks})$ in $F_k \mathcal{D}T(v+z)$ is nonzero.
Similarly,
$T(v+z-\delta^{ks}) \in \U_q \mathcal{D}T(v+z)$.
\end{proof}

For $v_{n-1,i}=v_{ni}, v_{n-1,j}=v_{nj}$ and $v_{ni}=v_{nj}$,
we provide a different construction.
Let $V(T(v))$ be the vector space spanned by the set of tableaux
\begin{align*}
\{T(v+z), \mathcal{D}T(v+z),& -(m-1)\leq z_{n-1,i},z_{n-1,j}\leq0,\\
z_{rs}&\in \mathbb Z \text{ for } (r,s)\notin\{ (n-1,i),(n-1,j)\} \}.
\end{align*}

We impose the following conditions on the tableaux:
$$T(v+z) = T(v+\tau(z))\quad \mathcal{D}T(v+z) + \mathcal{D}T(v+\tau(z)) = 0.$$
Where $\tau(z)$ is obtained from $z$ by switching $z_{n-1,i}$ and $z_{n-1,j}$ and all other entries are fixed.

\begin{theorem}\label{1-singular(2)}
$V(T(v))$ has a $\U_q$-module structure.
The action of generators of $\U_q$ is given by
\begin{align*}
& u(T(v+z))  = \mathcal{D}^{v}([x-y]_{q} u(T(v+z))), \\
&u(\mathcal{D} T({v}+z')) = \mathcal{D}^{v}(u(T(v+z'))).
\end{align*}
Moreover, for any $1\leq r\leq s \leq n$, the action of the generators of $\Gamma_q$ can be written as follows:
\begin{align*}
&c_{rs}(T(v+z)) = \gamma_{rs}(v+z)T((v+z)),\\
&c_{rs}(D(T(v+z))) =\gamma_{rs}(v+z)DT(v+z)+D^v\left(\gamma_{rs}(v+z)\right)T((v+z)).
\end{align*}
\end{theorem}
\begin{proof}
Replace $v$ by $v(x,y)$ where $v(x,y)_{rs} = v_{rs}$ if $(r,s)\notin \{(n,i), (n,j), (n-1,i), (n-1,j)\}$,
and $v(x,y)_{n,i}+x=v(x,y)_{n-1,i}+x$, $v(x,y)_{n,j}+y=v(x,y)_{n-1,j}+y$.
Then
\begin{multline*}
V\bigl(T(v(x,y))\bigr) = \operatorname{span}_{\mathbb{C}(x,y)}\bigl\{\, T(v(x,y)+z) \bigm|
-(m-1) \le z_{n-1,i},\, z_{n-1,j} \le 0, \\
z_{rs} \in \mathbb{Z} \text{ for all }(r,s) \notin \{(n-1,i),\,(n-1,j)\} \,\bigr\}
\end{multline*}
is a $\U_q$-module with action of the generators given by the formulas (\ref{GT-formulae}).

Let $L \subset V(T(v(x,y)))$ be the $\mathcal{F}$-span of $\{S(z), A(z)\}$,
where
$$S(z) = \frac{T(v(x,y)+z) + T(v(x,y)+\tau(z))}{2},$$
$$A(z) = \frac{T(v(x,y)+z) - T(v(x,y)+\tau(z))}{2[x-y]_q}.$$
Similar to the proof of Theorem~\ref{1-singular(1)},
we can equip $V(T(v))$ a  $\U_q(\mathfrak{gl}_n)$-module structure via the isomorphism
$L \otimes_{\mathcal{F}} \mathbb{C}\cong V(T(v))$.
The action of the central element $\Gamma_q$ follows similarly.
\end{proof}

Similar to \eqref{def of VT--1sig}, by taking subquotients of $V(T(v))$,
we can construct finite-dimensional modules.
By the same arguments as in the proof of Theorem~\ref{thm:pairwise-non-iso-modules},
we can show that for any fixed $v$ and distinct choices of $W_{rs}$ with $1\leq s\leq r\leq n-1$,
the modules $\overline{V}(T(v))$
are pairwise non-isomorphic.
Moreover,
these modules are not isomorphic to the modules $\overline{V}(T(v))$ in Theorem~ \ref{thm:pairwise-non-iso-modules}.

The actions of $E_k^m$ and $F_k^m$ on
$\overline{V}(T(v))$ in Theorem \ref{1-singular(1)} and \ref{1-singular(2)}
are scalar. For certain choice of $v$, the module is simple and these scalar actions are nonzero.
Consequently, we obtain cyclic modules with
Gelfand--Tsetlin weights corresponding to $1$-singular tableaux.

\section {Highest weight module over $\U_{q}(\mathfrak{gl}_3)$}\label{section:ex}
Using the constructions in Section \ref{section name:GT basis}, we obtain a family of cyclic modules and highest weight modules.
In this section, we employ the Gelfand--Tsetlin bases to study the baby Verma modules of $\U_q(\mathfrak{gl}_3)$ and determine the their composition series.

Let $\lambda=(\lambda_1,\lambda_2,\lambda_3)$ be a highest weight  of $\U_{q}(\mathfrak{gl}_3)$,
we assume that $\lambda_i-\lambda_j\in \mathbb{Z}$ or $\lambda_i-\lambda_j\notin\frac{\mathbb{Z}}{2}$.
Denote by $v$ the tableau:
\begin{center}
\Large	
	\begin{tabular}{c c}
		\xymatrixrowsep{0.2cm}
		\xymatrixcolsep{0.1cm}\xymatrix @C=0.2em {
			&\scriptstyle{v_{1}}    & &\scriptstyle{v_{2}}  & &\scriptstyle{v_{3}}  \\
			&&\scriptstyle{v_{1}}   & &\scriptstyle{v_{2}}   \\
			& &&\scriptstyle{v_{1}}   &     \\
		}
	\end{tabular}
\end{center}
where $v_1=\lambda_1$, $v_{2}=\lambda_2-1$, $v_3=\lambda_3-2$.
For any given $\lambda$, we construct certain baby Verma modules containing
Gelfand--Tsetlin weight associated with $v$. The construction differs depending on whether $\lambda_1 - \lambda_2 \notin \frac{\mathbb{Z}}{2}$ or $\lambda_1 - \lambda_2 \in \mathbb{Z}$.
For each case, we employ distinct methods
from Section~\ref{section name:GT basis}.

\subsection{Baby Verma modules for weights with $\lambda_{1}-\lambda_{2}\notin \frac{\mathbb{Z}}{2}$}\label{sec:baby_verma_noninteger_diff}
\numberwithin{equation}{section}
We can obtain such modules from the generic modules constructed in Section \ref{subsec:generic-GT-modules}.
By Theorem \ref{thm:structure}, for any fixed $v$ and distinct choices of $W_{kj}$ with $1\leq j\leq k\leq 2$,
the quotients $\overline{V}(T(v))$
are pairwise non-isomorphic.
In this section, we write $\overline{V}(T(v))$ as $V$ for short.
We now give $8$ pairwise non-isomorphic baby Verma modules and determine their irreducible subquotients.
Recall that the submodules of $\U_q(T(v))$ are defined as follows:
\begin{align*}
W_{21}^{(1)} &:= \langle\, T(v - m\delta^{21} - m\delta^{11}) - T(v) \,\rangle,
                &
W_{21}^{(2)} &:= \langle\, T(v - m\delta^{21} - m\delta^{11}) \,\rangle,\\[4pt]
W_{22}^{(1)} &:= \langle\, T(v - m\delta^{22}) - T(v) \,\rangle,
                &
W_{22}^{(2)} &:= \langle\, T(v - m\delta^{22}) \,\rangle,\\[4pt]
W_{11}^{(1)} &:= \langle\, T(v - m\delta^{11}) - T(v) \,\rangle,
                &
W_{11}^{(2)} &:= \langle\, T(v - m\delta^{11}) \,\rangle.
\end{align*}
(1) Let $V = \U_q(T(v)) \big/ W_{21}^{(2)}+W_{22}^{(2)}+W_{11}^{(2)}.$
A basis for $V$ is given by $T(v+z)$, where $z$ satisfies:
\begin{equation*}
\begin{aligned}
-(m-1)\leq z_{21},z_{22} \leq 0,\\
0\leq z_{21}-z_{11} \leq m-1,\\
z_{kj}\in \mathbb{Z},\; 1\leq j\leq k\leq 2.
\end{aligned}
\end{equation*}

\begin{enumerate}
\item[(1.1)]
If $\lambda_{1} - \lambda_{3} \notin \frac{\mathbb{Z}}{2}$ and $\lambda_{2} - \lambda_{3} \notin \frac{\mathbb{Z}}{2}$,
then $V$ is simple.

\item[(1.2)]
If $\lambda_{1} - \lambda_{3} \in \mathbb{Z}$,
then we may set $v_3 = v_1 - t$ with $0 \leq t \leq m-1$.\\
(1.2.1)
If $t = 0$, then $V$ is simple.\\
(1.2.2)
If $t \neq 0$, then $V$ admits the following composition series:
$$
0 \subsetneq \U_q T(v - t\delta_{21} - t\delta_{11}) \subsetneq V,
$$
where $\dim \U_q T(v - t\delta_{21} - t\delta_{11}) = m^2(m-t)$ and $\dim V = m^3$.
The highest weight of the composition factor $V \big/ \U_q T(v - t\delta_{21} - t\delta_{11})$ is $\lambda$, and the highest weight of $ \U_q T(v - t\delta_{21} - t\delta_{11})$ is
$(\lambda_3-2, \lambda_2, \lambda_1+2)$.

\item[(1.3)]
If $\lambda_{2} - \lambda_{3} \in \mathbb{Z}$, then we may set $v_3 = v_2 - t$ with $0 \leq t \leq m-1$.\\
(1.3.1)
If $t = 0$, then $V$ is simple.\\
(1.3.2)
If $t \neq 0$, then $V$ admits the following composition series:
$$
0 \subsetneq \U_q T(v - t\delta_{22}) \subsetneq V,
$$
where $\dim \U_q T(v - t\delta_{22}) = m^2(m-t)$ and $\dim V = m^3$.
The highest weight of the composition factor $V \big/ \U_q T(v - t\delta_{22})$ is $\lambda$, and the highest weight of $ \U_q T(v - t\delta_{22})$ is $(\lambda_1, \lambda_3-1, \lambda_2+1)$.
\end{enumerate}
(2) Let $V = \U_q(T(v)) \big/ W_{21}^{(1)}+W_{22}^{(2)}+W_{11}^{(2)}$.
A basis for $V$ is given by $T(v+z)$, where $z$ satisfies:
\begin{equation*}
\begin{aligned}
-(m-1)\leq z_{22} \leq 0,\\
0\leq z_{21}-z_{11} \leq m-1,\\
z_{21}\in \mathbb{Z}_m,\;
z_{11}, z_{22}\in \mathbb{Z}.
\end{aligned}
\end{equation*}
\begin{enumerate}
\item[(2.1)]
If $\lambda_{1} - \lambda_{3} \notin \frac{\mathbb{Z}}{2}$ and $\lambda_{2} - \lambda_{3} \notin \frac{\mathbb{Z}}{2}$,
then $V$ is simple.

\item[(2.2)]
If $\lambda_{1} - \lambda_{3} \in \mathbb{Z}$, then $V$ is simple.

\item[(2.3)]
If $\lambda_{2} - \lambda_{3} \in \mathbb{Z}$, then we may set $v_3 = v_2 - t$ with $0 \leq t \leq m-1$.\\
(2.3.1)
If $t = 0$, then $V$ is simple.\\
(2.3.2)
If $t \neq 0$, then $V$ admits the following composition series:
$$
0 \subsetneq \U_q T(v - t\delta_{22}) \subsetneq V,
$$
where $\dim \U_q T(v - t\delta_{22}) = m^2(m-t)$ and $\dim V = m^3$.
The highest weight of the composition factor $V \big/ \U_q T(v - t\delta_{22})$ is $\lambda$, and the highest weight of $ \U_q T(v - t\delta_{22})$ is $(\lambda_1, \lambda_3-1, \lambda_2+1)$.
\end{enumerate}
(3) Let $V = \U_q(T(v)) \big/ W_{21}^{(2)}+W_{22}^{(1)}+W_{11}^{(2)}$.
A basis for $V$ is given by $T(v+z)$, where $z$ satisfies:
\begin{equation*}
\begin{aligned}
-(m-1)\leq z_{21} \leq 0,\\
0\leq z_{21}-z_{11} \leq m-1,\\
z_{22}\in \mathbb{Z}_m,\;
z_{11}, z_{21}\in \mathbb{Z}.
\end{aligned}
\end{equation*}
\begin{enumerate}
\item[(3.1)]
If $\lambda_{1} - \lambda_{3} \notin \frac{\mathbb{Z}}{2}$ and $\lambda_{2} - \lambda_{3} \notin \frac{\mathbb{Z}}{2}$,
then $V$ is simple.

\item[(3.2)]
If $\lambda_{1} - \lambda_{3} \in \mathbb{Z}$, then we may set $v_3 = v_1 - t$ with $0 \leq t \leq m-1$.\\
(3.2.1)
If $t = 0$, then $V$ is simple.\\
(3.2.2)
If $t \neq 0$, then $V$ admits the following composition series:
$$
0 \subsetneq \U_q T(v - t\delta_{21} - t\delta_{11}) \subsetneq V,
$$
where $\dim \U_q T(v - t\delta_{21} - t\delta_{11}) = m^2(m-t)$ and $\dim V = m^3$.
The highest weight of the composition factor $V \big/ \U_q T(v - t\delta_{21} - t\delta_{11})$ is $\lambda$, and the highest weight of $ \U_q T(v - t\delta_{21} - t\delta_{11})$ is $(\lambda_3-2, \lambda_2, \lambda_1+2)$.

\item[(3.3)]
If $\lambda_{2} - \lambda_{3} \in \mathbb{Z}$, then $V$ is simple.
\end{enumerate}
(4)Let $V = \U_q(T(v)) \big/ W_{21}^{(1)}+W_{22}^{(1)}+W_{11}^{(2)}$.
A basis for $V$ is given by $T(v+z)$, where $z$ satisfies:
\begin{equation*}
\begin{aligned}
0\leq z_{21}-z_{11} \leq m-1,\\
z_{21},z_{22}\in \mathbb{Z}_m,\;
z_{11}\in \mathbb{Z}.
\end{aligned}
\end{equation*}
For any $\lambda_3$, $V$ is simple.\\
(5) Let $V = \U_q(T(v)) \big/ W_{21}^{(1)}+W_{22}^{(2)}+W_{11}^{(1)}$.
A basis for $V$ is given by $T(v+z)$, where $z$ satisfies:
\begin{equation*}
\begin{aligned}
-(m-1)\leq z_{22} \leq 0,\\
0\leq z_{21}-z_{11} \leq m-1,\\
z_{11},z_{21}\in \mathbb{Z}_m,\;
z_{22}\in \mathbb{Z}.
\end{aligned}
\end{equation*}
\begin{enumerate}
\item[(5.1)]
If $\lambda_{1} - \lambda_{3} \notin \frac{\mathbb{Z}}{2}$ and $\lambda_{2} - \lambda_{3} \notin \frac{\mathbb{Z}}{2}$,
then $V$ is simple.

\item[(5.2)]
If $\lambda_{1} - \lambda_{3} \in \mathbb{Z}$, then $V$ is simple.

\item[(5.3)]
If $\lambda_{2} - \lambda_{3} \in \mathbb{Z}$, then we may set $v_3 = v_2 - t$ with $0 \leq t \leq m-1$.\\
(5.3.1)
If $t = 0$, then $V$ is simple.\\
(5.3.2)
If $t \neq 0$, then $V$ admits the following composition series:
$$
0 \subsetneq \U_q T(v - t\delta_{22}) \subsetneq V,
$$
where $\dim \U_q T(v - t\delta_{22}) = m^2(m-t)$ and $\dim V = m^3$.
The highest weight of the composition factor $V \big/ \U_q T(v - t\delta_{22})$ is $\lambda$, and the highest weight of $ \U_q T(v - t\delta_{22})$ is $(\lambda_1, \lambda_3-1, \lambda_2+1)$.
\end{enumerate}
(6) Let $V = \U_q(T(v)) \big/ W_{21}^{(2)}+W_{22}^{(2)}+W_{11}^{(1)}$.
A basis for $V$ is given by $T(v+z)$, where $z$ satisfies:
\begin{equation*}
\begin{aligned}
-(m-1)\leq z_{21},z_{22} \leq 0,\\
0\leq z_{21}-z_{11} \leq m-1,\\
z_{11}\in \mathbb{Z}_m,\;
z_{21},z_{22}\in \mathbb{Z}.
\end{aligned}
\end{equation*}

\begin{enumerate}
\item[(6.1)]
If $\lambda_{1} - \lambda_{3} \notin \frac{\mathbb{Z}}{2}$ and $\lambda_{2} - \lambda_{3} \notin \frac{\mathbb{Z}}{2}$,
then $V$ is simple.

\item[(6.2)]
If $\lambda_{1} - \lambda_{3} \in \mathbb{Z}$, then we may set $v_3 = v_1 - t$ with $0 \leq t \leq m-1$.\\
(6.2.1)
If $t = 0$, then $V$ is simple.\\
(6.2.2)
If $t \neq 0$, then $V$ admits the following composition series:
$$
0 \subsetneq \U_q T(v - t\delta_{21} - t\delta_{11}) \subsetneq V,
$$
where $\dim \U_q T(v - t\delta_{21} - t\delta_{11}) = m^2(m-t)$ and $\dim V = m^3$.
The highest weight of the composition factor $V \big/ \U_q T(v - t\delta_{21} - t\delta_{11})$ is $\lambda$, and the highest weight of $ \U_q T(v - t\delta_{21} - t\delta_{11})$ is $(\lambda_3-2, \lambda_2, \lambda_1+2)$.

\item[(6.3)]
If $\lambda_{2} - \lambda_{3} \in \mathbb{Z}$, then we may set $v_3 = v_2 - t$ with $0 \leq t \leq m-1$.\\
(6.3.1)
If $t = 0$, then $V$ is simple.\\
(6.3.2)
If $t \neq 0$, then $V$ admits the following composition series:
$$
0 \subsetneq \U_q T(v - t\delta_{22}) \subsetneq V,
$$
where $\dim \U_q T(v - t\delta_{22}) = m^2(m-t)$ and $\dim V = m^3$.
The highest weight of the composition factor $V \big/ \U_q T(v - t\delta_{22})$ is $\lambda$, and the highest weight of $ \U_q T(v - t\delta_{22})$ is $(\lambda_1, \lambda_3-1, \lambda_2+1)$.
\end{enumerate}
(7) Let $V = \U_q(T(v)) \big/ W_{21}^{(2)}+W_{22}^{(1)}+W_{11}^{(1)}$.
A basis for $V$ is given by $T(v+z)$, where $z$ satisfies:
\begin{equation*}
\begin{aligned}
-(m-1)\leq z_{21} \leq 0,\\
0\leq z_{21}-z_{11} \leq m-1,\\
z_{11},z_{22}\in \mathbb{Z}_m,\;
z_{21}\in \mathbb{Z}.
\end{aligned}
\end{equation*}

\begin{enumerate}
\item[(7.1)]
If $\lambda_{1} - \lambda_{3} \notin \frac{\mathbb{Z}}{2}$ and $\lambda_{2} - \lambda_{3} \notin \frac{\mathbb{Z}}{2}$,
then $V$ is simple.

\item[(7.2)]
If $\lambda_{1} - \lambda_{3} \in \mathbb{Z}$, then we may set $v_3 = v_1 - t$ with $0 \leq t \leq m-1$.\\
(7.2.1)
If $t = 0$, then $V$ is simple.\\
(7.2.2)
If $t \neq 0$, then $V$ admits the following composition series:
$$
0 \subsetneq \U_q T(v - t\delta_{21} - t\delta_{11}) \subsetneq V,
$$
where $\dim \U_q T(v - t\delta_{21} - t\delta_{11}) = m^2(m-t)$ and $\dim V = m^3$.
The highest weight of the composition factor $V \big/ \U_q T(v - t\delta_{21} - t\delta_{11})$ is $\lambda$, and the highest weight of $ \U_q T(v - t\delta_{21} - t\delta_{11})$ is $(\lambda_3-2, \lambda_2, \lambda_1+2)$.

\item[(7.3)]
If $\lambda_{2} - \lambda_{3} \in \mathbb{Z}$, then $V$ is simple.
\end{enumerate}
(8) Let $V = \U_q(T(v)) \big/ W_{21}^{(1)}+W_{22}^{(1)}+W_{11}^{(1)}$.
A basis for $V$ is given by $T(v+z)$, where $z$ satisfies:
\begin{equation*}
\begin{aligned}
0\leq z_{21}-z_{11} \leq m-1,\\
z_{kj}\in \mathbb{Z}_m, 1\leq j\leq k\leq 2.
\end{aligned}
\end{equation*}
For any $\lambda_3$, $V$ is simple.

Since the highest weight $\lambda$ already determines part of the center, we now compute the actions of the remaining central elements
$F_1^m$, $F_2^m$ and $E_{31}^m$ on the modules $V$.
It suffices to compute their action on the highest weight vector $T(v)$.
By Proposition~\ref{prop:cyclic-module-criterion},
we have
\begin{equation}
\begin{aligned}
    &F_{1}^mT(v)=T(v-m\delta_{11}),\\
    &F_{2}^mT(v)=
    \prod_{k=1}^{m-1}
    \frac{[v_{2}-v_{1}-k]_q}
    {[v_{2}-v_{1}+k]_q}
    T(v-m\delta_{22}).
\end{aligned}
\end{equation}
Let
$E_{31}=F_2F_1-qF_1F_2$
(cf. \cite[Sections 6.2.2 and 6.2.3]{KS1997}),
then we have
\begin{align*}
&E_{31}^{m} T(v)
=\prod_{k=0}^{m-1}\frac{1}{[v_1-v_2-k]_q} \, T(v - m\delta_{21} - m\delta_{11}) \\
&-\prod_{k=1}^{m-1}
\left(\frac{[v_2-v_1+1]_q- q[v_2- v_1]_q}
{[v_1-v_2+k]_q} \right)
T(v - m\delta_{22} - m\delta_{11}).
\end{align*}
where both coefficients are nonzero.

The action of the central elements $F_1^m$, $F_2^m$ and $E_{31}^m$ on the modules $V$ is given as follows:
the action $F_1^m$ on $V$ is zero in cases $(1),(2),(3),(4)$ and a nonzero scalar in cases $(5),(6),(7),(8)$;
the action of $F_2^m$   on $V$ is zero in cases $(1),(2),(5),(6)$ and a nonzero scalar in cases  $(3),(4),(7),(8)$;
the action of $E_{31}^m$  on $V$ is zero in cases $(1),(2),(3),(4),(6)$ and a nonzero scalar in cases $(5),(7),(8)$.

\subsection{Baby Verma modules for weights with $\lambda_{1}-\lambda_{2}\in \mathbb{Z}$}
\label{sec:baby_verma_integer_diff}
For any given
$\lambda$ such that
$\lambda_{1}-\lambda_{2}\in \mathbb{Z}$, we construct two non-isomorphic modules  from the $1$-singular modules constructed in Section \ref{subsec:1-singular-GT-modules}.
Moreover, if $\lambda_1=\lambda_2-1$,
we have two additional non-isomorphic modules.
Recall that the submodules of
$\U_q(T(v))$ are defined as follows:
\begin{align*}
W_{11}^{(1)} &:= \langle\, T(v - m\delta^{11}) - T(v) \,\rangle,
                &
W_{11}^{(2)} &:= \langle\, T(v - m\delta^{11}) \,\rangle,
\end{align*}
and recall that $L(z)$ is defined by
$$
L(z): =
\begin{cases}
T(v+z), & z_{21}\ge z_{22},\\[4pt]
\mathcal{D}T(v+z), & z_{21}<z_{22}.
\end{cases}
$$
To simplify subsequent discussion, we set
$b: =v_2-v_1$ and $c: =v_3-v_1$,
where $v_1=\lambda_1$, $v_{2}=\lambda_2-1$, $v_3=\lambda_3-2$.

\subsubsection{}
The following two constructions are based on Theorem~\ref{1-singular(1)}.\\
(1) Let $V=\U_q(T(v)) \big/ W_{11}^{(1)}$.
A basis for $V$ is given by $L(z)$, where $z$ satisfies:
$$
z_{kj}\in \mathbb{Z}_m,\; 1\leq j\leq k\leq 2.
$$
For any $\lambda_{3}$, $N_1$ is simple.\\
(2) Let $V=\U_q(T(v)) \big/ W_{11}^{(2)}$.
A basis for $V$ is given by $L(z)$, where $z$ satisfies:
\begin{equation*}
\begin{aligned}
z_{21}-m<z_{11}\leq z_{21}, \\
z_{21},z_{22}\in \mathbb{Z}_m,\; z_{11}\in \mathbb{Z}.
\end{aligned}
\end{equation*}

\begin{enumerate}
\item[(2.1)]
$\lambda_{2} - \lambda_{3} \notin \frac{\mathbb{Z}}{2}$.\\
(2.1.1)
If $v_1=v_2$ (i.e., $b=0$), then the module $V$ is simple.\\
(2.1.2)
If $v_1\neq v_2$ (i.e., $b\neq 0$), then the module $V$ is not simple.
Let $J$ be the submodule of $V$ generated by $T(v + (0,0,b))$, the quotient $V/J$ is simple and has a basis
consisting of $L(z)$ such that:
\begin{equation*}
\begin{aligned}
 	-(m-1)\leq z_{21}\leq -m-b, \\
 	z_{21}-m< z_{11}\leq z_{21},\\
    z_{21},z_{22}\in \mathbb{Z}_m, z_{11}\in \mathbb{Z}.
\end{aligned}
\end{equation*}
The dimension of ${V/J}$ is
$m^2(m+b).$
The composition series of $V$ is as follows:
$$0\subsetneq J \subsetneq V,$$
where $\dim J=-b m^2$ (here $b=\lambda_2-\lambda_1-1<0$) and
and $\dim V = m^3$.
The highest weight of the composition factor $V \big/ J$ is $\lambda$, and the highest weight of $ J$ is $(\lambda_2-1, \lambda_1+1, \lambda_3)$.

\item[(2.2)]
$\lambda_{2} - \lambda_{3} \in \mathbb{Z}$. Set $d=\max\{b,c\}$, where $b$ and $c$ are taken from the set of representatives $\{-(m-1), \dots, 0\}$.\\
(2.2.1)
If $d = 0$, then $V$ is simple.\\
(2.2.2)
If $d\neq 0$, then $V$ is not simple.
Let $J$ be the submodule of $V$ generated by $T\bigl(v + (0,0,b)\bigr)$ for $d=b$, or by $T\bigl(v + (0,0,c)\bigr)$ for $d=c$,
the quotient $V/J$ is simple and has a basis
consisting of $L(z)$ such that:
\begin{equation*}
\begin{aligned}
 	-(m-1)\leq z_{21}\leq -m-d, \\
 	z_{21}-m< z_{11}\leq z_{21},\\
    z_{21},z_{22}\in \mathbb{Z}_m, z_{11}\in \mathbb{Z}.
\end{aligned}
\end{equation*}
The dimension of ${V/J}$ is
$m^2(m+d).$
The composition series of $V$ is as follows:
$$0\subsetneq J \subsetneq V,$$
where $\dim J=-dm^2$ and
and $\dim V = m^3$.
The quotient $V/J$ has highest weight $\lambda$.
For $d=b$, the highest weight of $J$ is $(\lambda_2-1, \lambda_1+1, \lambda_3)$;
for $d=c$, it is $(\lambda_3-2, \lambda_1+1, \lambda_2+1)$.
\end{enumerate}

\subsubsection{}
If $\lambda_1 = \lambda_2-1$ (i.e., $b=0$),
we have two additional non-isomorphic modules by
Theorem~\ref{1-singular(2)}.\\
(3) Let $V=\U_q(T(v)) \big/ W_{11}^{(1)}$.
A basis for $V$ is given by $L(z)$, where $z$ satisfies:
\begin{equation*}
\begin{aligned}
-(m-1)\leq z_{21},z_{22}\leq0, \\
z_{11}\in \mathbb{Z}_m,\;
z_{21},z_{22}\in \mathbb{Z}.
\end{aligned}
\end{equation*}

\begin{enumerate}
\item[(3.1)]
If $\lambda_{2} - \lambda_{3} \notin \frac{\mathbb{Z}}{2}$, then $V$ is simple.

\item[(3.2)]
$\lambda_{2} - \lambda_{3} \in \mathbb{Z}$.\\
(3.2.1)
If $c = 0$, then $V$ is simple.\\
(3.2.2)
If $c\neq 0$, then $V$ is not simple.
Let $J$ be the submodule of $N_3$ generated by $T(v + (0,c,c))$, the quotient $V/J$ is simple and has a basis
consisting of $L(z)$ such that:
\begin{equation*}
\begin{aligned}
 	-(m-1)\leq z_{22}\leq c, \\
 	z_{21}-m< z_{11}\leq z_{21},\\
     z_{11}\in \mathbb{Z}_m,\; z_{21},z_{22}\in \mathbb{Z}.
\end{aligned}
\end{equation*}
The dimension of ${V/J}$ is
$m^2(m+c).$
The composition series of $V$ is as follows:
$$0\subsetneq J \subsetneq V,$$
where $\dim J=-cm^2$ and
and $\dim V = m^3$.
The highest weight of the composition factor $V \big/ J$ is $\lambda$, and the highest weight of $J$ is $(\lambda_3-2, \lambda_1+1, \lambda_2+1)$.
\end{enumerate}

(4) Let $V=\U_q(T(v)) \big/ W_{11}^{(2)}$.
A basis for $V$ is given by $L(z)$, where $z$ satisfies:
\begin{equation*}
\begin{aligned}
-(m-1)\leq z_{21},z_{22}\leq0, \\
z_{21}-m<z_{11}\leq z_{21}, \\
z_{kj}\in \mathbb{Z},\;
1\leq j\leq k\leq 2.
\end{aligned}
\end{equation*}

\begin{enumerate}
\item[(4.1)]
If $\lambda_{2} - \lambda_{3} \notin \frac{\mathbb{Z}}{2}$, then $V$ is simple.

\item[(4.2)]
$\lambda_{2} - \lambda_{3} \in \mathbb{Z}$.\\
(4.2.1)
If $c = 0$, then $V$ is simple.\\
(4.2.2)
If $c\neq 0$, then $V$ is not simple.
Let $J$ be the submodule of $V$ generated by $DT(v+(c,0,-m))$ and the maximal vector $T(v+(0,c,0))$,
the quotient module $V/J$ is simple and has a basis consisting of tableaux $L(z)$ such that
\begin{equation*}
\begin{aligned}
\left(  \left \{ \begin{array}{c}
	c< z_{21}\leq 0 \\
	c< z_{22}\leq 0 \\
	z_{21}-m< z_{11}\leq z_{21}\\
    z_{kj}\in \mathbb{Z}, 1\leq j\leq k\leq 2
\end{array}\right \} \bigcup
\left \{ \begin{array}{c}
	c<z_{22}\leq0 \\
	-(m-1)\leq z_{21}\leq c\\
	z_{22}-m< z_{11}\leq z_{21}\\
    z_{kj}\in \mathbb{Z}, 1\leq j\leq k\leq 2.
\end{array}\right \} \right)\\
\end{aligned}
\end{equation*}
The dimension of ${V/J}$ is
$\dfrac{mc(c-m)}{2}.$
The composition series of $V$ is as follows:
$$0\subsetneq \U_qT(v+(0,c,c))
\subsetneq \U_qT(v+(0,c,0)) \subsetneq J \subsetneq V.$$
The highest weight of the composition factor $V \big/ J$ is $\lambda$. The highest weights of both $J \big/ \U_qT(v+(0,c,0))$ and $\U_qT(v+(0,c,0)) \big/ \U_qT(v+(0,c,c))$ are $(\lambda_1, \lambda_3-1, \lambda_2+1)$, and the highest weight of $\U_qT(v+(0,c,c))$ is $(\lambda_3-2, \lambda_1+1, \lambda_2+1)$.
\end{enumerate}

Since the highest weight $\lambda$ already determines part of the center, we now compute the actions of the remaining central elements
$F_1^m$, $F_2^m$ and $E_{31}^m$ on the modules $V$.
It suffices to compute their action on the highest weight vector $T(v)$.
Direct computation yields
\begin{align*}
    &F_{1}^mT(v)=T(v-m\delta_{11}).\\
    &F_{2}^mT(v)=\begin{cases}
T(v-m\delta_{21}), & \text{if } v_1=v_2, \\
T(v-m\delta_{22}), & \text{if } v_1\neq v_2.
\end{cases}\\
\end{align*}
Recall that $E_{31}=F_2F_1-qF_1F_2$.
A direct calculation gives
\begin{align*}
E_{31}^m T(v) =\frac{m(q^{-1}-q)}{2[m-1]_q!} T(v-m\delta_{11}).
\end{align*}
where the coefficient is nonzero.

The central elements $F_1^m$, $F_2^m$ and $E_{31}^m$ act as follows:
the action $F_1^m$ on $V$ is zero in cases $(2),(4)$ and a nonzero scalar in cases $(1),(3)$;
the action of $F_2^m$ on $V$ is zero in cases $(3),(4)$ and a nonzero scalar in cases $(1),(2)$;
the action $E_{31}^m$ on $V$ is zero in cases $(2),(4)$ and a nonzero scalar in cases $(1),(3)$.

\bigskip
\noindent
\textbf{Acknowledgments}
This work is supported by the Natural Science Foundation of China (No. 12571026),
the Natural Science Foundation of Hubei Province (No. 2025AFB716) and the Fundamental Research Funds for the Central
Universities
(No. XJ2026002101).

\end{document}